\documentclass[12pt]{amsart}
\usepackage{amsmath, amsthm, amscd, amsfonts, amssymb, graphicx, color,float,pgf,tikz}
\usepackage{amssymb,fontenc}
\usepackage{latexsym,wasysym,mathrsfs}
\usepackage{hyperref}
\usetikzlibrary{arrows}
\makeatletter \oddsidemargin.9375in \evensidemargin \oddsidemargin
\newtheorem{theorem}{Theorem}[section]
\newtheorem{lemma}[theorem]{Lemma}

\theoremstyle{definition}
\newtheorem{definition}[theorem]{Definition}

\theoremstyle{remark}

\numberwithin{equation}{section}

\newcommand{\be}{\begin{equation}}
\newcommand{\ee}{\end{equation}}

\numberwithin{equation}{section}
\theoremstyle{definition}

\theoremstyle{definition}
\theoremstyle{remark}

\numberwithin{equation}{section}

\allowdisplaybreaks
\begin{document}

\noindent {}   \\[0.50in]


\title[The split common null point problem] {The split common null point problem for generalized resolvents and nonexpansive mappings in Banach spaces}

\author [Orouji and Soori]{Bijan Orouji$^{1}$, Ebrahim Soori$^{2,*}$}
\thanks{$^{*}$ Corresponding authorbijanorouji@yahoo.com(B. Orouji);
sori.e@lu.ac.ir(E. Soori)
}





\maketitle
\hrule width \hsize \kern 1mm

\begin{abstract}
In this paper,    the split common null point problem in two Banach spaces is considered. Then, using the generalized resolvents of maximal monotone operators and the generalized projections and an infinite family of nonexpansive mappings,    a strong convergence theorem for finding a solution of the split common null point problem in two Banach spaces in the presence of a sequence of errors  will be proved.

\end{abstract}
\maketitle
\vspace{0.1in}
\hrule width \hsize \kern 1mm

\section{\textbf{Introduction}}
Let $H_1$ and $H_2$ be two Hilbert spaces and  $C$ and $Q$,   two nonempty, closed and convex subsets of $H_1$ and $H_2$, respectively. Let $A:H_1\rightarrow H_2$ be a bounded linear operator. Then the \textit{split feasibility problem} (SFP) \cite{p6} is:   to find $z\in H_1$ such that $z\in C\cap A^{-1}Q$. There exists several generalizations of the SFP: the multiple set convex sets problem ( MSSFP)  \cite{p19, p7}, the split common fixed point problem (SCFPP) \cite{p8,p20}, and the split common null point problem (SCNPP) \cite{p5}. (SCNPP) is as follows: given set-valued mappings $M_1:H_1\rightarrow 2^{H_1}$ and $M_2:H_2\rightarrow 2^{H_2}$, and a bounded linear operator $A:E\rightarrow F$, find a point $z\in H_1$ such that
\begin{equation*}
  z\in M_1^{-1}0\cap A^{-1}(M_2^{-1}0),
\end{equation*}
where $M_1^{-1}0$ and $M_2^{-1}0$ are sets of null points of $M_1$ and $M_2$, respectively.
 Many authors have studied the split feasibility problem and the split common null point problem using  nonlinear operators and fixed points; see, for example, \cite{p3,p5,p8,p10,p19,p26,p20,p35}. However, we have not found many results outside of the framework  Hilbert spaces. Note that the first extension of SFP to Banach spaces is  appeared  in \cite{p28}, then  this scheme  was later extended to MSSFP in \cite{p36}. A very recent generalization for the SFP is  appeared in\cite{p29}. The split common null point problem in Banach spaces is also  solved by Takahashi \cite{p32,p33,p34}.

In this paper,  the split common null point problem with generalized resolvents of maximal monotone operators in two Banach spaces is considered. Then using the generalized resolvents of maximal monotone operators and the generalized projections, a strong convergence theorem for finding a solution of the split null point problem in two banach spaces in the presence of a sequence of errors  is proved.

\section{\textbf{Preliminaries}}
Let $ E $ be a real Banach space with the  norm $ \Vert. \Vert $ and   $E^{*}$,   the dual space of $E$. When $\lbrace  x_{n} \rbrace $ is a sequence in $E$,   the strong convergence   of $\lbrace x_{n}\rbrace $ to $ x\in E $ is denoted  by $x_{n}\rightarrow x$ and the weak convergence to $ x\in E $ is denoted  by $x_{n}\rightharpoonup x$. A Banach space $E$ is strictly convex if $ \Vert \frac{x+y}{2} \Vert <1$, whenever $x,y \in S(E)$,  $x\neq y $ and $S(E)$ is the unite sphere centered at  the origin of $E$. $E$ is said to be uniformly convex if $\delta_{E}(\epsilon)=0$ and $ \delta_{E}(\epsilon)>0$ for all $0<\epsilon\leq 2$ where $\delta_{E}(\epsilon)$ is the modulus of convexity of $E$ and is defined by
\begin{align}
\delta_{E}(\epsilon)=inf \bigg\lbrace 1-\frac{\Vert x+y \Vert}{2}:\Vert x\Vert,\Vert y\Vert\leq 1,\Vert x-y\Vert \geq \epsilon \bigg\rbrace
\end{align}
A uniformly convex Banach space is strictly convex and reflexive. It is also well known that a uniformly convex Banach space has Kadec Klee property, that is, $x_{n}\rightharpoonup u $ and $\Vert x_{n}\Vert \rightarrow \Vert u \Vert $ imply $ x_{n} \rightarrow u$, see \cite{p11,p23}. Furthermore, $E$ is called $p$-uniformly convex if there exists a constant $ c>0 $ such that $ \delta_{E}  \geq c\epsilon^{p} $ for all $ \varepsilon\in [0,2]$, where $p$ is a fixed real number with $ p\geq 2$. For example, the $ L_{p}$ space is $2$-uniformly convex for $1<p\leq 2$ and $p$-uniformly convex for $p\geq 2$, see   \cite{p25}. Let $U=\lbrace x\in E:\Vert x \Vert =1 \rbrace$. The norm of $E$ is said to be Gateaux differentiable if for each $x,y\in U$, the limit
\begin{align}
\lim _{t_{\rightarrow}0}\frac{\Vert x+ty\Vert-\Vert x\Vert}{t}
\end{align}
exists. In this case, $E$ is called smooth. The modulus of smoothness of $E$ is defined by
\begin{align}
\rho _{E}(t)= \sup \bigg\lbrace \frac{\Vert x+y\Vert +\Vert x-y\Vert}{2}-1 : x\in U,\Vert y\Vert \leq t\bigg\rbrace.
\end{align}
If $\displaystyle\lim_{t\rightarrow 0}\frac{\rho _{E}(t)}{t}=0$, then $E$ is called uniformly smooth. Let $q>1$. If there exists a fixed constant $c>0$ such that $\rho_{E}(t)\leq ct^{q}$, then $E$ is said to be $q$-uniformly smooth, see \cite{jouy}. It is well known that a uniformly convex Banach space is reflexive and strictly convex.

 The mapping $ J_{E}^{p}$ from $E$ to $2^{E^{*}}$ is  defined by
\begin{align}
J_{E}^{p}(x)=\lbrace x^{*} \in E^{*}:\langle x,x^{*} \rangle=\Vert x \Vert \Vert x^{*}\Vert ,\Vert x^{*}\Vert =\Vert x\Vert ^{p-1}\rbrace, \quad \forall x\in E
\end{align}
If $p=2$ then $ J_{E}^{2}=J_{E}$ is the normalized duality mapping on $E$.  Note that $E$ is smooth if and only if $J_{E}$ is a single-valued mapping of $E$ into $E^{*}$. We also know that $E$ is reflexive if and only if $J_{E}$ is surjective, and $E$ is strictly convex if and only if $J_{E}$ is one-to-one. Hence, if $E$ is smooth, strictly convex and reflexive Banach space, then $J_{E}$ is a single-valued, bijection and in this case, the inverse mapping $J_{E}^{-1}$ coincides with the duality mapping $J_{{E}^{*}}: E^{*}\rightarrow 2^{E}$, that means $ J_{E}^{-1}=J_{{E}^{*}}$. If $E$ is uniformly convex and uniformly smooth, then is uniformly norm-to-norm continuous on bounded sets of $E$ and $ J_{E}^{-1}=J_{{E}^{*}}$ is also uniformly norm-to-norm continuous on bounded sets of $E^{*}$. It is known that $E$ is $p$-uniformly convex if and only if its dual $E^{*}$ is $q$-uniformly smooth where $1<q\leq 2\leq p<\infty$ with $ \frac{1}{p}+\frac{1}{q}=1$. For more details about  the mapping $J_{E}^{p} $  refer to  \cite{ag, p17,p25, p11,p12,p23,p30,p31}.
\begin{lemma}\label{qww} \cite{p10q}
Let $x,y\in E$ if $E$ is $q$-uniformly smooth, then there is a $c_{q}>0$ so that
\begin{equation*}
\Vert x-y\Vert^{q} \leq \Vert x\Vert^{q}-q\langle y, J_{E}^{q}(x)\rangle +c_{q}\Vert y\Vert^{q}.
\end{equation*}
\end{lemma}
Suppose that $E$ is a smooth Banach space and   $J$ is the duality mapping on $E$. Define a function $\phi :E\times E\rightarrow \mathbb{R}$ by
\begin{equation}
\phi_{E}(x,y)=\Vert x\Vert^{2}-2\langle x,Jy\rangle +\Vert y\Vert^{2} \quad \forall x,y\in E.
\end{equation}
Observer that, in a Hilbert space H, $\phi(x,y)=\Vert x-y\Vert^{2}$ for all $ x,y\in H$. Furthermore, we know that for each $x,y,z,\omega \in E$,
\begin{equation}\label{iokl}
(\Vert x \Vert -\Vert y \Vert )^{2}\leq \phi (x,y)\leq (\Vert x \Vert +\Vert y \Vert )^{2};
\end{equation}
\begin{equation}\label{whyu}
\phi \Big(t,J^{-1}\big(\lambda Jx+(1-\lambda )Jy\big)\Big) \leq \lambda \phi(t,x)+(1-\lambda )\phi (t,y);
\end{equation}
\begin{equation}
2\langle x-y,Jz-J\omega \rangle = \phi(x,\omega )+\phi(y,z)-\phi(x,z)-\phi(y,\omega ).
\end{equation}
if $E$ is   additionally assumed to be a strictly convex Banach space, then
\begin{equation*}
\phi(x,y)=0 \quad \Longleftrightarrow \quad x=y.
\end{equation*}
The following lemma is due to Kamimura and Takahashi \cite{p16}.
\begin{lemma}\cite{p16}\label{qee}
Let $E$ be a uniformly convex and smooth Banach space and let $\lbrace y_{n}\rbrace ,\lbrace z_{n}\rbrace$ be two sequences of $E$, if $ \displaystyle\lim_{n\rightarrow \infty}\phi (y_{n},z_{n})=0$ and either $\lbrace y_{n}\rbrace $ or $\lbrace z_{n}\rbrace$ is bounded, then $\displaystyle\lim_{n\rightarrow \infty}(y_{n}-z_{n})=0$.
\end{lemma}
Suppose that  $C$ is a nonempty, closed and convex subset of a smooth, strictly convex, and reflexive Banach space $E$. Then  for any $x\in E$, there exists a unique element $ z\in C$ such that
\begin{equation}
\phi (z,x)=\min_{y\in C}\phi (y,x).
\end{equation}
The mapping $\Pi_C :E\rightarrow C$ defined by $z=\Pi_C x$ is called the generalized projection of $E$ onto $C$. For example, see \cite{p1,p2,p16}.
\begin{lemma}\cite{p37}\label{qrr}
Let $E$ be a smooth, strictly convex and reflexive Banach space and $C$ be a nonempty, closed and convex subset of $E$ and let $ x \in E $ and $ z\in C $. Then, the following condition hold:
\begin{flushleft}
 \begin{itemize}
  \item [{\rm($1$)}] $\phi (z,\Pi_Cx)+\phi (\Pi_Cx,x)\leq \phi(z,x)\quad \forall x\in C,y\in E $;
  \item [{\rm($2$)}] $z=\Pi_Cx\Longleftrightarrow \langle y-z,Jx-Jz\rangle \leq 0\quad \forall y\in C$.
  \end{itemize}
\end{flushleft}
\end{lemma}
Suppose that  $M$ is a mapping of $E$ into $2^{E^{*}}$ for the Banach space $E$. The effective domain of $M$ is denote by $ dom(M) $, that is, $ dom(M)=\lbrace x\in E: Mx\neq \emptyset \rbrace $. A multi-valued mapping $M$ on $E$ is said to be monotone if $\langle x-y, u^{*}-v^{*}\rangle\geq 0 $ for all $ x,y \in dom(M), u^{*} \in Mx$ and $v^{*}\in My $.
A monotone operator $M$ on $E$ is said to be maximal if it's graph is not property contained in the graph of any other monotone operator on $E$.
\begin{theorem}\cite{p4,p24}\label{qtt}
Let $E$ be a uniformly convex and smooth Banach space and let $J$ be the duality mapping of $E$ into $E^*$. Let $M$ be a monotone operator of $E$ into $2^{E^{*}}$. Then $M$ is maximal if and only if for any $r>0$,
\begin{equation*}
R(J+rM)=E^{*},
\end{equation*}
where $R(J+rM)$ is the range of $J+rM$.
\end{theorem}
Let $E$ be a uniformly convex Banach space with a Gateaux differentiable norm and let $M$ be a maximal monotone operator of $E$ into $2^{E^{*}}$. For all $x\in E$ and $r>0$, we consider the following equation
\begin{equation*}
Jx\in Jx_r+rMx_r
\end{equation*}
This equation has a unique solution $x_r$. In fact, it is obvious from Theorem 3.1 that there exists a solution $x_r$ of $Jx\in Jx_r+rMx_r$. Assume that $Jx\in Ju+rMu$ and $Jx\in Jv+rMv$. Then there exist $\omega_1\in Mu$ and $\omega_2 \in Mv$ such that $Jx=Ju+r\omega_1$ and $Jx=Jv+r\omega_2$. So, we have that
\begin{align*}
0&=\langle u-v, Jx-Jx\rangle \\ &=\langle u-v, Ju+r\omega_1-(Jv+r\omega_2)\rangle \\
&=\langle u-v, Ju-Jv+r\omega_1-r\omega_2\rangle \\ &=\langle u-v, Ju-Jv\rangle + \langle u-v, r\omega_1- r\omega_2\rangle \\ &=\phi (u,v)+\phi (v,u)+r\langle u-v,\omega_1-\omega_2\rangle \\ &\geq \phi (u,v)+\phi (v,u)
 \end{align*}
and hence $0=\phi(u,v)=\phi(v,u)$. Since $E$ is strictly convex, we have $u=v$. We defined $J_r^M$ by $x_r=J_r^Mx$, such that $J_r^M, r>0$ are called the generalized resolvents of $M$. The set of null points of $M$ is defined by $M^{-1}0=\lbrace z\in E:0\in Mz\rbrace $. We know that $M^{-1}0$ is closed and convex; see[26]. Furthermore
\begin{align}\label{ghyyy}
 \langle J_r^Mx-y,J(x-J_r^Mx)\rangle \geq 0,
 \end{align}
  hold for all $x\in E$ and $y\in M^{-1}0$;see \cite{p26}.
\begin{lemma}\cite{p9}\label{qyy}
Let $E$ be a real reflexive, strictly convex and smooth Banach space, $M:E\rightarrow 2^{E^{*}}$ be a maximal monotone with $M^{-1}0\neq \emptyset$, then for any $x\in E , y\in M^{-1}0$ and $r>0$, we have
\begin{equation*}
\phi(y,J_r^Mx)+\phi(J_r^Mx,x)\leq \phi(y,x).
\end{equation*}
where $J_r^M:E\rightarrow E$ is defined by $J_r^M:=(J+rM)^{-1}J$.
\end{lemma}
\begin{definition}\cite{p22}\label{qggh}
   $M$ is called upper semicontinuous if for any closed subset $C$ of $E^*$, $M^{-1}(C)$ is closed.
\end{definition}
\begin{theorem}\cite{p22}\label{qttl}
Let  $M : E \rightarrow 2^{E^*}$ be a maximal monotone operator with $dom(M)=E$. Then, $M$ is upper semicontinuous.
\end{theorem}
 A Banach space $E$ is said to satisfy the Opial condition, if whenever a sequence $\lbrace x_n\rbrace$ in $E$ converges weakly to $x_0\in E$, then
\begin{equation}
\liminf_{n\rightarrow \infty}\Vert x_n-x_0\Vert <\liminf_{n\rightarrow \infty}\Vert x_n-x\Vert,\quad \forall x\in E , x\neq x_0
\end{equation}
\begin{definition}\label{quu}
Let $C$ be a nonempty convex subset of a Banach space, $\{T_{i}\}_{i \in \mathbb{N}}$ a sequence of nonexpansive mappings of $C $ into itself
and $\{\lambda_{i}\}$ a real sequence such that $0\leq\lambda_{i}\leq 1$ for every $i\in \mathbb{N}$. Following {\rm \cite{p18}}, for any $n \geq1$, we define a mapping\;$ W_{n}$
of $C $ into itself as follows,
\begin{align}
&U_{n,n+1}:=I,\nonumber\\
&U_{n,n}:=\lambda_{n}T_{n}U_{n,n+1}+(1-\lambda_{n})I,\nonumber\\
&\:\;\vdots\nonumber\\
&U_{n,k}:=\lambda_{k}T_{k}U_{n,k+1}+(1-\lambda_{k})I,\label{3}\\
&\:\;\vdots\nonumber\\
&U_{n,2}:=\lambda_{2}T_{2}U_{n,3}+(1-\lambda_{2})I,\nonumber\\
W_{n}:=&U_{n,1}:=\lambda_{1}T_{1}U_{n,2}+(1-\lambda_{1})I,\nonumber
\end{align}
\end{definition}
The following results hold for the mappings $ W_{n}$.
\begin{theorem}\cite{p18}\label{2.8}
Let $C$ be a nonempty closed convex subset of a strictly convex Banach space. Let $\{T_{i}\}_{i \in \mathbb{N}} $ be a sequence of
nonexpansive mappings of $C$ into itself such that $\bigcap_{i=1}^{\infty} \rm{Fix}(T_{i})\neq\emptyset$ and let $\{\lambda_{i}\}$ be a real sequence such that $ 0\leq \lambda_{i}\leq b < 1 $ for every $ i \in \mathbb{N}$. For any $n\in \mathbb{N}$, let $W_n$ be the $W$-mapping of $C$ into itself generated by $T_n, T_{n-1}, . . . , T_1$ and $\lambda_n, \lambda_{n-1}, . . . , \lambda_1$. Then
 \begin{itemize}
\item [{\rm(1)}] $ W_{n}$ is asymptotically regular and nonexpansive and $ \rm{Fix}(W_{n})=\bigcap_{i=1}^{n}\rm{Fix}(T_{i})$, for all $n \in \mathbb{N}$.
\item [{\rm(2)}] for every $x \in C $ and for each positive integer $j$, $\displaystyle\lim_{n\rightarrow\infty} U_{n,j}x$ exists.
\item [{\rm(3)}] The mapping $ W: C\rightarrow C$ defined by
$Wx:=\displaystyle\lim_{n\rightarrow\infty}W_{n}x=\displaystyle\lim_{n\rightarrow\infty}U_{n,1}x$, for every $x\in C,$
is a nonexpansive mapping satisfying $\rm{Fix}(W)=\bigcap_{i=1}^{\infty}\rm{Fix}(T_{i})$ and it is called the $W$-mapping generated by $\{T_{i}\}_{i\in\mathbb{N}}$, and $\{\lambda_{i}\}_{i\in\mathbb{N}}$.
 \end{itemize}
\end{theorem}
\begin{theorem}\cite{p15}\label{2.9}
Let $\{T_{i}\}_{i=1}^{\infty} $ be a sequence of nonexpansive mappings of $C$ into itself such that
$\bigcap_{i=1}^{\infty} \rm{Fix}(T_{i})\neq \emptyset,\{\lambda_{i}\}$ be a real sequence such that $0 < \lambda_{i}\leq b < 1, \;( i \geq1)$. If $D$ is any bounded subset
of $ C$, then
$\displaystyle\lim_{n\rightarrow\infty}\sup_{x\in D}\|Wx -W_{n}x\|=0$.
\end{theorem}
\begin{lemma}\cite{p3q}\label{muko}
  Let $E$ be a strictly convex Banach space and $C\subseteq E$ be a nonempty and convex subset of $E$. let $T:C\rightarrow E$ be a nonexpansive mapping. Then $Fix(T)$ is convex.
\end{lemma}

\section{\textbf{Main results}}
\begin{theorem}\label{aaa}
Let $E$ and $F$ be two $2$-uniformly convex and uniformly smooth real Banach spaces that satisfy the Opial condition. Let $J_E$ and $J_F$ be the duality mappings on $E$ and $F$, respectively. Suppose that $C$ is nonempty, closed and convex subset of $E$.
Let $A:E\longrightarrow F$ be a bounded linear operator such that $A\neq0$ with the adjoint operator $ A^*$.
 Let $M_1$ be a maximal monotone operator of $E$ into $2^{E^{*}}$ such that $M_1^{-1}0\neq\emptyset$ and $M_2$ be a maximal monotone operator of $F$ into $2^{ F^{*}}$ such that $M_2^{-1}0\neq\emptyset$. Suppose that $S:C\longrightarrow E$ is a nonexpansive mapping and $\lbrace T_i \rbrace_{i=1}^{\infty} :C\longrightarrow C$, a family of nonexpansive mappings. For every $n\in \mathbb{N}$, let $W_n$ be a $W-mapping$ generated by Definition \ref{quu}. Let $J_{\lambda}^{M_1}$ and $Q_{\mu}^{M_2}$ be the generalized resolvents of $M_1$  and $M_2$ for $\lambda>0 $ and $\mu>0$, respectively. Suppose that $M_1^{-1}0 \cap A^{-1}(M_2^{-1}0)\subseteq C$. Let $x_1\in C$ and let $\lbrace x_{n}\rbrace$, $\lbrace u_{n}\rbrace$ and $\lbrace y_{n}\rbrace$ be the sequences generated by
  \begin{align}\label{mnasd}
  \begin{cases}
  u_n=J_{E}^{-1}\Big((1-\alpha_n)J_{E}x_n+\alpha_nJ_{E}\Pi_{C}J_{E}^{-1}(\sigma_nJ_{E}W_n x_n+(1-\sigma_n)J_{E}Sx_n)\Big);\\
  z_n=J_{\lambda_n}^{M_1}(u_n+e_n);\quad \quad \omega_n=Q_{\mu_n}^{M_2}(Az_n);\\
  y_n =\Pi_C J_{E}^{-1}\Big(J_{E}z_n-\gamma A^*J_{F}(Az_n-\omega_n)\Big);\\
  C_n =\lbrace z\in C,\big\langle \omega_n-Az,J_{F}(Az_n-\omega_n)\big\rangle\geq 0\rbrace;\\
  D_n =\lbrace z\in E,\phi_{E}(z,z_n)\leq \phi_{E}(z,u_n+e_n)\rbrace;\\
  Q_n =\lbrace z\in E,\langle x_n-z,J_{E}x_1-J_{E}x_n\rangle\geq 0\rbrace;\\
  x_{n+1}=\Pi_{C_n\cap Q_n\cap D_n}x_1,\quad \forall n\in \mathbb{N}.
  \end{cases}
  \end{align}
 where $J_{\lambda_n}^{M_1}=(J_{E}+\lambda_n M_1)^{-1}J_{E}$ and $Q_{\mu_n}^{M_2}=(J_{F}+\mu_n M_2)^{-1}J_{F}$ such that $\lbrace\lambda_n\rbrace,\lbrace\mu_n\rbrace \subseteq(0,\infty)$ and $a\in \mathbb{R}$ satisfy in $0<a\leq \lambda_n,\mu_n,\forall n\in \mathbb{N}$ and $0<\gamma<\frac{2}{c\parallel A\parallel^{2}}$ with $c>0$.
   Let $\lbrace\alpha_n\rbrace,\lbrace\sigma_n\rbrace$ be real sequences in $(0,1)$ satisfied in the conditions:
  \begin{itemize}
  \item [{\rm(i)}] $\displaystyle\lim_{ n\rightarrow \infty}\alpha_n = 0$;
  \item [{\rm(ii)}] $\displaystyle\lim_{n\rightarrow \infty}\dfrac{\Vert J_Ex_n-J_E u_n\Vert}{\alpha_n}=0$;
  \item [{\rm(iii)}] $\displaystyle\lim_{n\rightarrow \infty}\sigma_n =1$.
  \end{itemize}
   Consider the error sequence $\lbrace e_n\rbrace \subseteq E$ such that
  \begin{itemize}
    \item [{\rm(iv)}] $\displaystyle\lim_{n\rightarrow \infty}\Vert e_n\Vert =0$.
  \end{itemize}
   Let $\Omega = M_{1}^{-1} 0 \cap A^{-1}(M_{2}^{-1} 0 )\cap\big(\bigcap_{i=1}^{\infty}Fix(T_i)\big)$.
  Suppose that one of the following two conditions   holds:
  \begin{itemize}
    \item [(v)] the sequence $\{x_n\}$ is bounded,\\
    or
    \item [(vi)] $\Omega \neq \emptyset$.
  \end{itemize}
  Then
  \begin{enumerate}
    \item [(a)]$\Omega \neq \emptyset$ if and only if the sequence $\{x_n\}$ is bounded,
    \item [(b)]the sequence $\lbrace x_n\rbrace$ converges strongly to a point $\omega_0\in \Omega$ where $\omega_0 =\Pi_\Omega x_1$.
  \end{enumerate}
  \end{theorem}

  \begin{proof}
  (a)  Let $\Omega \neq \emptyset$. First, it will be checked that $C_{n}\cap D_{n}\cap Q_{n}$ is closed and convex for all $ n\in \mathbb{N}$. For any $z\in D_n$, it is realized that
  \begin{align}\label{asgh}
  &\phi_E(z,z_n)\leq \phi_E(z,u_n+e_n)\nonumber\\
  \Leftrightarrow &\| z\|^2+\|z_n\|^2-2\langle z,J_Ez_n\rangle \leq \|z\|^2+\|u_n+e_n\|^2-2\langle z,J_E(u_n+e_n)\rangle \nonumber \\
  \Leftrightarrow &\|u_n+e_n\|^2-\|z_n\|^2+2\langle z,J_Ez_n\rangle -2\langle z,J_E(u_n+e_n)\rangle \geq 0.
  \end{align}
  Because $E$ is a real Banach space, the inner product of $E$ is linear in both components and jointly continuous. Therefore, it is easily  observed from (\ref{asgh}) that $D_n$ is closed and convex for all $n \in \mathbb{N}$. Further, since $A$ is a bounded linear operator, it is obvious that $C_n$ is closed and convex for all $n\in \mathbb{N}$. Also, it is evident that $Q_n$ is closed and convex for all $n \in \mathbb{N}$. Consequently, $C_n\cap D_n\cap Q_n$ is closed and convex for all $n \in \mathbb{N}$.\\
  \indent Now, it will be shown that $ M_{1}^{-1}0\cap A^{-1}(M_{2}^{-1}0)\subseteq C_{n}$ for all $ n\in \mathbb{N}$.
  In fact, since $Q_{\mu_{n}}^{M_{2}}$ be the resolvent of $M_2$, we have from \eqref{ghyyy} for all $z\in A^{-1}(M_2^{-1}0)$ that
  \begin{equation*}
  \langle Q_{\mu_n}^{M_2}Az_n-Az,J_{F}(Az_n-Q_{\mu_n}^{M_2}Az_n)\rangle\geq 0,\quad\forall n\in \mathbb{N},
  \end{equation*}
  then $M_{1}^{-1}0\cap A^{-1}(M_2^{-1}0)\subseteq C_n$ for all $n\in \mathbb{N}$. Next, it will be demonstrated that $M_{1}^{-1}0\cap A^{-1}(M_2^{-1}0)\subseteq D_n$ for all $n\in \mathbb{N}$. Indeed, since $M_1$ is maximal monotone, hence from Lemma \ref{qyy}, it is implied for all $z\in M_1^{-1}0$ and $\lambda_n >0$ that
  \begin{equation*}
  \phi_{E} (z,z_n )=\phi_{E}\big(z,J_{\lambda_n}^{M_1}(u_n+e_n)\big)\leq \phi_{E}(z, u_n+e_n),
  \quad\forall n\in \mathbb{N},
  \end{equation*}
   hence, $ M_1^{-1}0 \subseteq D_n $ for all $ n\in \mathbb{N} $, and therefore  $  M_{1}^{-1}0\cap A^{-1}(M_{2}^{-1}0)\subseteq D_{n} $ for all $ n\in \mathbb{N} $. Now, it will be shown by induction that $  M_{1}^{-1}0\cap A^{-1}(M_{2}^{-1}0)\subseteq Q_{n} $ for all $ n\in \mathbb{N} $. Since  $ \langle x_1 -z , J_{E}x_1 - J_{E}x_1 \rangle\geq 0 $ for all $ z\in E $, it is obvious that  $  M_{1}^{-1}0\cap A^{-1}(M_{2}^{-1}0)\subseteq Q_1 = E $. Suppose that $ M_{1}^{-1}0\cap A^{-1}(M_{2}^{-1}0)\subseteq Q_k $ for some $k\in\mathbb{N}$. Then $  M_{1}^{-1}0\cap A^{-1}(M_{2}^{-1}0)\subseteq C_k\cap D_k  \cap Q_k $. From the fact that $ x_{k+1} = \Pi_{C_k\cap D_K\cap Q_k}x_1 $, it is implied from Lemma \ref{qrr}  that
 \begin{equation*}
 \langle x_{k+1} -z , J_{E} x_1 - J_{E} x_{k+1} \rangle\geq 0,\quad  \forall z \in C_k\cap D_k \cap Q_k .
 \end{equation*}
  Since $  M_{1}^{-1}0\cap A^{-1}(M_{2}^{-1}0)\subseteq C_k\cap D_k \cap Q_k $, it is concluded that
 \begin{equation*}
 \langle x_{k+1} -z , J_{E} x_1 - J_{E} x_{k+1} \rangle\geq 0,\quad \forall z\in M_{1}^{-1}0\cap A^{-1}(M_{2}^{-1}0).
 \end{equation*}
     So $ M_{1}^{-1}0\cap A^{-1}(M_{2}^{-1}0)\subseteq Q_ {k+1} $. Hence by induction, $ M_{1}^{-1}0\cap A^{-1}(M_{2}^{-1}0)\\\subseteq Q_ {n} $ for all $ n\in \mathbb{N} $. Thus
   \begin{equation*}
    M_{1}^{-1}0\cap A^{-1}(M_{2}^{-1}0) \subseteq C_n \cap D_n \cap Q_n,\quad \forall n\in \mathbb{N}.
   \end{equation*}
   Hence $C_n \cap D_n \cap Q_n \neq \emptyset$. This implies that $\lbrace x_n \rbrace $ is well defined.

   By Theorem \ref{qttl} and Definition \ref{qggh}, the set $ M_{1}^{-1}0\cap A^{-1}(M_{2}^{-1}0)$ is closed. Next, we show that $M_{1}^{-1}0$ and $ A^{-1}(M_{2}^{-1}0)$ are convex subsets of $E$. Indeed, it is observed that for all $x_1 , x_2 \in M_{1}^{-1}0$ and for all $t \in [0,1]$
   \begin{align*}
   \langle 0-v , tx_1+(1-t)x_2-u\rangle =&t\langle 0-v , x_1-u\rangle +(1-t)\langle 0-v , x_2-u\rangle \\
    &\geq 0, \quad\forall u \in Dom(M_1) , v \in M_1u,
   \end{align*}
   hence, by the maximal monotonicity of $M_1$, it is implied that $tx_1+(1-t)x_2 \in M_1^{-1}0$. Also, for all $x_1 , x_2 \in A^{-1}(M_2^{-1}0)$ and for all $t \in [0,1]$ we have
   \begin{align*}
   \langle 0-v,tAx_1+(1-t)Ax_2-u\rangle =&t\langle 0-v,Ax_1-u\rangle +(1-t)\langle 0-v,Ax_2-u\rangle \\
   &\geq 0, \quad\forall u \in Dom(M_2) , v \in M_2u,
   \end{align*}
    and since $A$ is linear then by the maximal monotonicity of $M_2$, it is implied that  $tx_1+(1-t)x_2 \in A^{-1}(M_2^{-1}0)$.
    Thus, $M_{1}^{-1}0\cap A^{-1}(M_{2}^{-1}0)$ is a convex subset of $E$ and therefore, by Lemma \ref{muko}, $\Omega$ is closed and convex.

    Since $\Omega$ is a nonempty, closed and convex subset of $E$, there exists $\omega_0 \in \Omega$ such that
   $ \omega_0 =\Pi_\Omega x_1$. Since  $x_{n+1} = \Pi _{C_n \cap D_n \cap Q_n } x_1 $, it is concluded that
   \begin{equation*}
   \phi_E (x_{n+1} , x_1 )\leq \phi_E (y,x_1 ),\quad \forall y\in C_n \cap D_n \cap Q_n,
   \end{equation*}
    and since $\omega_0 \in \Omega \subseteq \ M_{1}^{-1}0\cap A^{-1}(M_{2}^{-1}0) \subseteq C_n \cap D_n \cap Q_n$, it is observed that
   \begin{equation}\label{zzz}
   \phi_E (x_{n+1} , x_1 )\leq \phi_E (\omega_0,x_1 ).
   \end{equation}
    This means that $\lbrace x_n \rbrace $ is bounded.

    Conversely, suppose $\lbrace x_n \rbrace $ is bounded. First, we show that $\displaystyle\lim_{n\rightarrow \infty }\phi_E (x_{n+1},x_n )=0$.
    Since $x_{n+1} \in C_n \cap D_n \cap Q_n $ hence
   \begin{align}\label{qssd}
   \phi_E(x_{n+1} ,x_n)&=\phi_E(x_{n+1},\Pi _{C_{n-1}\cap D_{n-1}\cap Q_{n-1}}x_1)\nonumber\\
   &\leq  \phi _E(x_{n+1},x_1)-\phi_E(\Pi_{C_{n-1}\cap D_{n-1}\cap Q_{n-1}}x_1,x_1)\nonumber\\
   &=\phi_E(x_{n+1},x_1)-\phi_E(x_{n},x_1).
   \end{align}
   and hence
   $
   \phi_E (x_{n} , x_1 ) + \phi_E (x_{n+1} , x_n ) \leq \phi_E (x_{n+1} , x_1 ),
   $
   so
   \begin{equation}\label{qaa}
   \phi_E (x_{n} , x_1 ) \leq \phi_E (x_{n+1} , x_1 ),
   \end{equation}
   therefore, it is concluded from \eqref{iokl}, \eqref{qaa} that $ \lbrace \phi _E (x_n , x_1 ) \rbrace $ is bounded and nondecreasing.
   Then, there exists the limit of $ \lbrace \phi _E (x_n , x_1 ) \rbrace. $ Using (\ref{qssd}), it is realized that
   \begin{equation}\label{tyop}
   \lim_{n\longrightarrow \infty } \phi_E (x_{n+1} , x_n )=0.
   \end{equation}
   From Lemma \ref{qee}, it is implied that
   \begin{equation}\label{qss}
   \lim_{n\longrightarrow \infty } \Vert x_{n+1} - x_n \Vert =0 .
   \end{equation}

    Next, we show that $\displaystyle\lim_{n\rightarrow\infty}\phi_E(x_{n+1},u_n)=0$. From the inequality \eqref{whyu}, it is implied that
    \begin{align*}
    \phi_E ( x_{n+1},u_n)=&\phi_E \big(x_{n+1},J_{E}^{-1}\big((1-\alpha_n)J_{E}x_n\\&+\alpha_n J_{E}\Pi_{C} J_{E}^{-1}(\sigma_n J_{E}W_n x_n+ (1-\sigma_n )J_{E}Sx_n)\big)\big)\\
    \leq &(1-\alpha_n )\phi_E( x_{n+1},x_n)\\&+\alpha_n\phi_E\big(x_{n+1} , \Pi_{C} J_{E}^{-1}(\sigma_n J_{E}W_nx_n+(1-\sigma_n)J_{E}Sx_n)\big)\\
    \leq &(1-\alpha_n )\phi_E ( x_{n+1},x_n)\\&+\alpha_n\phi_E\big(x_{n+1} , J_{E}^{-1}(\sigma_n J_{E}W_n x_n+(1-\sigma_n )J_{E}Sx_n)\big)\\
     \leq &(1-\alpha_n )\phi_E ( x_{n+1},x_n)\\&+\alpha_n \sigma_n\phi_E ( x_{n+1},W_n x_n)+\alpha_n(1-\sigma_n )\phi_E ( x_{n+1},Sx_n),
    \end{align*}
     and since from \eqref{iokl}, $\lbrace\phi_E(x_{n+1},W_n x_n)\rbrace$ and $\lbrace\phi_E(x_{n+1},Sx_n)\rbrace $ are bounded. hence, using $ (i)$ and \eqref{tyop},  it is concluded that
  \begin{equation}\label{qgg}
  \lim_{n\rightarrow\infty}\phi_E(x_{n+1},u_n) =0.
  \end{equation}
  Therefore, it is realized from Lemma \ref{qee} that
  \begin{equation}\label{qdd}
  \lim_{n\rightarrow\infty} \Vert x_{n+1}-u_n\Vert =0.
 \end{equation}
 Hence, it is followed from (\ref{qss}) and (\ref{qdd}) that
 \begin{equation}\label{qff}
 \lim_{n\rightarrow\infty} \Vert x_n-u_n\Vert =0.
 \end{equation}
 Furthermore, since $E$ is uniformly smooth and $J_E$ is uniformly continuous, it is implied from \eqref{qff} that
 \begin{equation}
 \lim_{n\rightarrow\infty} \Vert J_E x_n-J_E u_n\Vert =0.
 \end{equation}

 Also, since $x_{n+1}\in D_n$ then
 \begin{equation}\label{qhh}
 \phi_E(x_{n+1},z_n)\leq \phi_E(x_{n+1},u_n+e_n)
 \end{equation}
 Notice that
 \begin{align}
 \phi_E(x_{n+1},u_n+e_n)-\phi_E(x_{n+1},u_n)=&\Vert u_n+e_n\Vert -\Vert u_n\Vert\nonumber \\
 &+2\langle x_{n+1},J_Eu_n-J_E(u_n+e_n)\rangle
 \end{align}
 Since $J_E$ is uniformly continuous on each bounded subset of $E$ and $\displaystyle\lim_{n\rightarrow \infty}\Vert e_n\Vert =0$, we know from (\ref{qgg}) and (\ref{qhh}) that $\displaystyle\lim_{n\rightarrow \infty}\phi_E(x_{n+1},u_n+e_n)=0$, which implies that
 \begin{equation}\label{qjj}
 \lim_{n\rightarrow \infty}\phi_E(x_{n+1},z_n)=0.
 \end{equation}
 Therefore, it is implied from Lemma \ref{qee} that
  \begin{equation}\label{qkk}
  \lim_{n\rightarrow\infty} \Vert x_{n+1}-z_n\Vert =0.
 \end{equation}
 Hence, it is followed from \eqref{qss} and \eqref{qkk} that
 \begin{equation}\label{qxx}
 \lim_{n\rightarrow\infty} \Vert x_n-z_n\Vert =0.
 \end{equation}
  From \eqref{qff} and \eqref{qxx} it is implied that
 \begin{equation}\label{qcc}
 \lim_{n\rightarrow\infty}\Vert u_n-z_n\Vert =0.
 \end{equation}

 Let $ z\in C_n \cap D_n \cap Q_n$. Using Lemma \ref{qww}, it is concluded that
     \begin{align}\label{qvpl}
     \phi_E(z,y_n )=&\phi_E\big(z,\Pi_C J_{E}^{-1} \big(J_{E}z_n -\gamma A^*J_{F}(Az_n -\omega_n)\big)\big)\nonumber\\
      \leq &\phi_E \big(z , J_{E}^{-1}\big(J_{E}z_n -\gamma A^*J_{F}(Az_n -\omega_n)\big)\big)\nonumber\\
      =&\Vert z\Vert^{2}+\Vert J_{E}z_n-\gamma A^*J_{F}(Az_n -\omega_n )\Vert^{2}\nonumber\\
      &-2\langle z,J_{E}z_n-\gamma A^*J_{F}(Az_n -\omega_n)\rangle \nonumber\\
      =&\Vert z\Vert^{2}+\Vert J_{E}z_n-\gamma A^*J_{F}(Az_n-\omega_n)\Vert^{2}-2\langle z,J_Ez_n \rangle \nonumber\\
      &+2\gamma \langle z , A^*J_{F}(Az_n -\omega_n)\rangle \nonumber\\
     \leq &\Vert z\Vert^{2}+\Vert J_E z_n\Vert^{2}-2\gamma \langle Az_n,J_F(Az_n -\omega_n)\rangle\nonumber\\
     &+c\gamma^{2}\Vert A\Vert^{2}\Vert Az_n -\omega_n\Vert^2 \notag -2\langle z,J_Ez_n\rangle+2\gamma \langle Az,J_F(Az_n-\omega_n)\rangle \notag \nonumber\\
     =&\phi_E(z,z_n)-2\gamma \langle Az_n-Az,J_F(Az_n-\omega_n)\rangle\nonumber\\
     &+c\gamma^{2}\Vert A\Vert^{2}\Vert Az_n-\omega_n\Vert^{2} \notag\\
     =&\phi_E(z,z_n)-2\gamma \langle Az_n-\omega_n,J_F(Az_n-\omega_n)\rangle \notag\\
     &-2\gamma \langle \omega_n-Az,J_F(Az_n-\omega_n)\rangle+c\gamma^{2} \Vert A\Vert^{2}\Vert Az_n-\omega_n\Vert^{2} \notag \nonumber\\
     =&\phi_E(z,z_n)-2\gamma \Vert Az_n-\omega_n\Vert^{2}-2\gamma\langle\omega_n-Az,J_F(Az_n-\omega_n)\rangle \notag \nonumber\\
     &+c\gamma^{2}\Vert A\Vert^{2}\Vert Az_n-\omega_n\Vert^{2} \notag \nonumber\\
     \leq &\phi_E(z,z_n)-\gamma (2-c\gamma\Vert A\Vert^{2})\Vert Az_n-\omega_n \Vert^{2},
    \end{align}
    hence, it is implied from the condition $ 0< \gamma < \dfrac{2}{c \Vert A \Vert ^{2}}$ that
     \begin{equation}\label{qvv}
     \phi_E(z,y_n )\leq \phi_E(z,z_n).
     \end{equation}
 Also, using \eqref{qvv} and \eqref{qjj}, it is concluded that
 \begin{equation}\label{qzz}
 \lim_{n\rightarrow\infty}\phi_E(x_{n+1},y_n) =0.
 \end{equation}
 Therefore, it is followed from \eqref{qss}, \eqref{qzz} and Lemma \ref{qee} that
 \begin{equation}\label{qbb}
 \lim_{n\rightarrow\infty}\Vert x_n-y_n\Vert =0.
 \end{equation}

 Next, it is evaluated that $\displaystyle\lim_{n\rightarrow\infty}\Vert Az_n -\omega_n\Vert =0.$ Indeed, putting $z=x_{n+1}$, from (\ref{qvpl}), it is observed  that
 \begin{equation}
 \gamma (2-c\gamma \Vert A\Vert^{2})\Vert Az_n -\omega_n\Vert^{2}\leq \phi_E(x_{n+1},z_n)-\phi_E(x_{n+1} ,y_n)\leq  \phi_E(x_{n+1},z_n),
 \end{equation}
 for all $ n\in \mathbb{N} $. Since $ 0<\gamma<\frac{2}{c\parallel A\parallel^{2}}$, it is concluded from \eqref{qjj} that
 \begin{equation}\label{qnn}
 \lim_{n\rightarrow\infty}\Vert Az_n-\omega_n\Vert =0.
 \end{equation}
 Since $\lbrace x_n \rbrace $ is bounded, from \eqref{qff}, \eqref{qxx}, \eqref{qbb} and \eqref{qnn}, the sequences
 $\lbrace u_n\rbrace,$ $\lbrace z_n\rbrace,$ $\lbrace y_n\rbrace $ and $\lbrace \omega_n\rbrace $ are bounded.

  From the condition (iii) and the fact that $\{\Vert J_E Sx_n-J_E W_n x_n\Vert \}$ is bounded, it is observed that
 \begin{align*}
 \lim_{n\rightarrow\infty} \Vert &\big(\sigma_n J_E W_n x_n+(1- \sigma_n)J_E Sx_n \big)-J_E W_n x_n \Vert \nonumber\\
 &= \lim_{n\rightarrow\infty} \Vert (1- \sigma_n)J_E Sx_n-(1- \sigma_n)J_E W_n x_n\Vert \nonumber\\
  &= \lim_{n\rightarrow\infty} (1- \sigma_n)\Vert J_E Sx_n-J_E W_n x_n\Vert =0,
 \end{align*}
 hence, from the continuity of $J_E\Pi_CJ_E^{-1}$, it is understood that
 \begin{align}\label{wrr}
 &\lim_{n\rightarrow\infty}\Vert J_E\Pi_C J_E^{-1}\big(\sigma_nJ_E W_n x_n+(1-\sigma_n)J_ESx_n \big)-J_E\Pi_CJ_E^{-1}(J_E W_n x_n)\Vert\nonumber\\
 &=\lim_{n\rightarrow\infty} \Vert J_E \Pi_C J_E^{-1} \big(\sigma_n J_E W_n x_n+(1- \sigma_n)J_E Sx_n  \big)-J_E(W_n x_n)\Vert =0.
 \end{align}
Now, it is shown that $\displaystyle\lim_{n\rightarrow\infty}\Vert J_E u_n - J_E W_n x_n \Vert=0$,
 \begin{align*}
 \Vert J_Eu_n-J_E W_n x_n\Vert =&\big\Vert \big((1-\alpha_n)J_{E}x_n\\
 &+\alpha_n J_{E}\Pi_{C} J_{E}^{-1}\big(\sigma_n J_{E}W_n x_n+ (1-\sigma_n )J_{E}Sx_n)\big)\\
 &-J_E W_n x_n \big\Vert \\
 =&\big\Vert \big((1-\alpha_n)J_{E}x_n -(1-\alpha_n)J_E W_n x_n \big)\\
 &+\alpha_n \big( J_{E}\Pi_{C} J_{E}^{-1}\big( \sigma_n J_{E}W_n x_n+ (1-\sigma_n )J_{E}Sx_n \big)\\
 &-J_E W_n x_n \big)\big\Vert \\
 \leq&(1-\alpha_n)\Vert J_{E}x_n-J_E u_n \Vert \\
 &+(1-\alpha_n)\Vert J_{E}u_n-J_E W_n x_n \Vert \\
 &+\alpha_n\big\Vert \big( J_{E}\Pi_{C} J_{E}^{-1}\big( \sigma_n J_{E}W_n x_n+(1-\sigma_n )J_{E}Sx_n \big)\\
 &-J_E W_n x_n \big)\big\Vert,
 \end{align*}
 which implies that
 \begin{align*}
 \alpha_n\Vert J_Eu_n-J_EW_nx_n \Vert \leq& (1-\alpha_n)\Vert J_{E}x_n-J_E u_n \Vert\\
 &+\alpha_n\big\Vert \big( J_{E}\Pi_{C} J_{E}^{-1}\big(\sigma_n J_{E}W_n x_n+(1-\sigma_n )J_{E}Sx_n \big)\\
 &-J_E W_n x_n \big)\big\Vert,
 \end{align*}
 therefore,
 \begin{align*}
  \Vert J_Eu_n-J_EW_nx_n\Vert \leq& (1-\alpha_n)\frac{\Vert J_{E}x_n-J_Eu_n\Vert}{\alpha_n}+\big\Vert \big( J_{E}\Pi_{C} J_{E}^{-1}\big(\sigma_n J_{E}W_n x_n\\
  &+(1-\sigma_n)J_{E}Sx_n\big)-J_EW_n x_n\big)\big\Vert
 \end{align*}
 now, using \eqref{wrr} and the condition (ii), it is concluded that
 \begin{equation*}
 \lim_{n\rightarrow\infty}\Vert J_E u_n - J_E W_n x_n \Vert=0,
 \end{equation*}
 hence, because $E^{*}$  is uniformly smooth, it is induced that
 \begin{equation*}
 \lim_{n\rightarrow\infty}\Vert u_n -W_n x_n \Vert=0,
 \end{equation*}
 therefore, it is deduced from \eqref{qff} that
 \begin{equation}\label{wuu}
 \lim_{n\rightarrow\infty}\Vert x_n-W_nx_n\Vert=0.
 \end{equation}

 Since $\lbrace x_n \rbrace $ is bounded, there exists a subsequence $\lbrace x_{n_k} \rbrace $ which converges weakly to a point $ x^{*} \in E$. First, we show that $ x^{*}\in\bigcap_{i=1}^{\infty}Fix(T_i)$. To see that, by Theorems  \ref{2.8} and \ref{2.9},  the mapping $W : C \rightarrow C$ satisfies
\begin{equation}\label{15}
    \lim_{n\rightarrow \infty}\|W_{n}x^{*}- Wx^{*}\|=0.
\end{equation} Moreover, from Theorem \ref{2.8}, it is followed that $ Fix(W)=\cap_{i=1}^{\infty}Fix(T_i) $. Assume that $ x^{*}\notin\cap_{i=1}^{\infty}Fix(T_i)$ then $x^{*}\neq Wx^{*}$ and using \eqref{wuu}, \eqref{15} and Opial's property of Banach space, it is concluded that
 \begin{align*}
 \liminf_{k\rightarrow \infty}\Vert x_{n_k}-x^{*}\Vert <&\liminf_{k\rightarrow \infty}\Vert x_{n_k}-Wx^{*}\Vert \\
 \leq &\liminf_{k\rightarrow \infty}\big(\Vert x_{n_k}-W_{n_k}x_{n_k}\Vert +
 \Vert W_{n_k}x_{n_k}-W_{n_k}x^{*}\Vert\\
 &+\Vert W_{n_k}x^{*}-Wx^{*}\Vert \big)\\
 \leq &\liminf_{k\rightarrow \infty}\Vert x_{n_k}-x^{*}\Vert.
 \end{align*}
 which is a contradiction. Therefore, $ x^{*}\in\cap_{i=1}^{\infty}Fix(T_i)$.

  Next, it will be checked that $x^{*}\in M_{1}^{-1}0$. From \eqref{qff} and the fact that $\lbrace x_{n_k}\rbrace$ converges weakly to $x^{*}$, there exists a subsequence $\lbrace u_{n_k}\rbrace$ of $\lbrace u_n\rbrace$ converging weakly to $x^{*}$
  and therefore from (\ref{qcc}), it is induced that $\lbrace J_{\lambda_{n_k}}^{M_1}(u_{n_k}+e_{n_k})\rbrace$ converges
  weakly to $x^{*}$. Also, from $(iv)$ and \eqref{qcc}, it is implied that
  \begin{equation*}
  \lim_{n\rightarrow\infty}\Vert (u_n+e_n)-J_{\lambda_n}^{M_1}(u_n+e_n)\Vert =0,
  \end{equation*}
    hence, since $E$ is uniformly smooth, it is understood that
  \begin{equation}\label{wpp}
  \lim_{n\rightarrow\infty} \Vert J_E(u_n+e_n)-J_E J_{\lambda_n}^{M_1}(u_n+e_n)\Vert =0.
  \end{equation}
  Since $ J_{\lambda_n}^{M_1}$ is the generalized resolvent of $ M_1$, it is observed that
  \begin{equation*}
  \frac{J_E(u_n+e_n)-J_E J_{\lambda_n}^{M_1}(u_n+e_n)}{{\lambda_n}}\in M_1 J_{\lambda_n}^{M_1}(u_n+e_n),\quad\forall n\in \mathbb{N}.
  \end{equation*}
  From the monotonicity of $M_1$, it is deduced that
  \begin{equation}\label{opst}
  \Big\langle r- J_{\lambda_{n_k}}^{M_1}(u_{n_k}+e_{n_k}) , t^{*}- \frac{J_E(u_{n_k}+e_{n_k})-J_E J_{\lambda_{n_k}}^{M_1}(u_{n_k}+e_{n_k})}{{\lambda_{n_k}}} \Big\rangle \geq 0,
  \end{equation}
 for all $(r,t^{*})\in M_1$. From (\ref{wpp}) and the condition $0<a\leq \lambda_{n_k}$, it is followed that
 $\langle r-x^{*},t^{*}-0\rangle \geq 0,$ for all $(r,t^{*})\in M_1.$ Since $M_1$ is maximal monotone, we have
 $x^{*}\in M_{1}^{-1}0$.

  Next, we show that $x^{*}\in A^{-1}(M_{2}^{-1}0)$. From \eqref{qxx} and the fact that $\lbrace x_{n_k}\rbrace$ converges weakly to $x^{*}$, there exists a subsequence $\lbrace z_{n_k}\rbrace$ of $\lbrace z_n\rbrace$ converging weakly to $x^{*}$
  and since $A$ is bounded and linear, we also have that $\lbrace Az_{n_k}\rbrace$ converges weakly to $Ax^{*}$. Therefore, from (\ref{qnn}), we have $\lbrace Q_{\mu_{n_k}}^{M_2}Az_{n_k}\rbrace$ converges weakly to $Ax^{*}$.
  Also, since $F$ is uniformly smooth, it is induced from (\ref{qnn}) that
  \begin{equation}\label{waa}
  \lim_{n\rightarrow\infty} \Vert J_F Az_n-J_F\omega_n\Vert =0.
  \end{equation}
   Since $Q_{\mu_n}^{M_2}$ is the generalized resolvent of $M_2$,
   it is understood that
  \begin{equation*}
  \frac{J_FAz_n-J_F Q_{\mu_n}^{M_2}Az_n}{{\mu_n}} \in M_2 Q_{\mu_n}^{M_2} Az_n, \quad \forall n\in  \mathbb{N}.
  \end{equation*}
  From the monotonicity of $M_2$, it follows that
  \begin{equation}\label{quos}
  \Big\langle b- Q_{\mu_{n_k}}^{M_2}Az_{n_k} , f^{*}- \frac{J_FAz_{n_k}-J_F Q_{\mu_{n_k}}^{M_2}Az_{n_k}}{{\mu_{n_k}}}  \Big\rangle \geq 0, \quad \forall (b,f^{*}) \in M_2.
  \end{equation}
  From (\ref{waa}) and the condition $0<a\leq \mu_{n_k}$, it is concluded that $\langle b-Ax^{*},f^{*}-0\rangle \geq 0,$ for all $(b,f^{*})\in M_2$. Since $M_2$ is maximal monotone, it is implied that $x^{*}\in A^{-1}(M_{2}^{-1}0)$.\\
  Therefore, $x^{*}\in \Omega = M_{1}^{-1}0\cap A^{-1}(M_{2}^{-1}0)\cap\big(\bigcap_{i=1}^{\infty}Fix(T_i)\big)$ hence, $\Omega \neq \emptyset $.

 (b) Now, let $\bar{x}$ be an arbitrary element of $\omega_\omega (x_n)$ (the set of all weak limit point of the sequence $\{x_n\}$). Then there exists another subsequence $\lbrace x_{n_i} \rbrace$ of $\lbrace x_n\rbrace$  which converges weakly to  $\bar{x}$. Clearly, repeating the same argument, it is implied that $\bar{x}\in \Omega$. It is claimed that $\bar{x}=x^{*}$. Indeed, suppose that
  $\bar{x}\neq x^{*}$. Obviously, from \eqref{qss}, the sequences $\{x_n\}$ is cauchy and hence the sequences $\{\Vert x_n-\bar{x}\|\}$ and $\{\Vert x_n-x^*\|\}$ are convergent. Again, using Opial's property, it is concluded that
  \begin{align*}
  \lim_{n\rightarrow\infty}\Vert x_n-\bar{x}\|&=\liminf_{i\rightarrow\infty}\Vert x_{n_i}-\bar{x}\Vert <\liminf_{i\rightarrow\infty}\Vert x_{n_i}-x^*\Vert=\lim_{n\rightarrow\infty}\Vert x_n-x^*\| \\ &=\liminf_{k \rightarrow\infty}\Vert x_{n_k}-x^{*}\| <\liminf_{k\rightarrow\infty}\Vert x_{n_k}-\bar{x}\Vert=\lim_{n\rightarrow\infty}\Vert x_n-\bar{x}\|,
  \end{align*}
  this is a contradiction and thus $\bar{x}= x^{*}$. Therefore, $\omega_\omega (x_n)$ is singleton. Thus $\lbrace x_n\rbrace$ converges weakly to $x^*\in\Omega$.
  Since norm is weakly lower semicontinuous, it is implied from \eqref{zzz} that
  \begin{align*}
  \phi_E (\omega_0,x_1) &= \phi_E (\Pi_\Omega x_1,x_1)\leq \phi_E (x^{*},x_1)\qquad\qquad \\
  &= \Vert x^{*}\Vert ^{2} - 2\langle x^{*},J_Ex_1\rangle+\Vert x_1\Vert ^{2} \\
  &\leq \liminf _{k\rightarrow\infty}(\Vert x_{n_k}\Vert ^{2} - 2\langle x_{n_k},J_Ex_1\rangle+\Vert x_1\Vert ^{2})\\
  &= \liminf _{k\rightarrow\infty} \phi_E (x_{n_k},x_1)\\
  &\leq \limsup _{k\rightarrow\infty} \phi_E (x_{n_k},x_1)\leq \phi_E (\omega_0,x_1),
  \end{align*}
  hence, from the definition of $\Pi_\Omega x_1$, it is understood that $\omega_0=x^{*}$ and
  \begin{equation*}
  \lim_{k\rightarrow\infty}\phi_E(x_{n_k},x_1)=\phi_E(x^{*},x_1)=\phi_E(\omega_0,x_1).
  \end{equation*}
  So, it is deduced that $\displaystyle\lim_{k\rightarrow\infty}\Vert x_{n_k}\Vert =\Vert \omega_0\Vert$. From the Kadec-Klee property of $E$, it is concluded that $\displaystyle\lim_{k\rightarrow\infty} x_{n_k} = \omega_0$. Therefore, $\displaystyle\lim_{k\rightarrow\infty} x_n = \omega_0$. Thus $\lbrace x_n\rbrace$ converges strongly to $x^{*}$ where $x^{*}=\Pi_\Omega x_1$.
\end{proof}

\section{\textbf{Applications and numerical example}}
In this section, using Theorem \ref{aaa}, a new strong convergence theorem in Banach spaces will be demonstrated. Let $E$ be a Banach space and $f$ be a proper, lower semicontinuous and convex function of $E$ into $(-\infty ,\infty]$.\\ Recall the definition of the subdifferential $\partial f$ of $f$ as follows:
\begin{equation*}
\partial f(x)=\lbrace z^*\in E^*:f(x)+\langle y-x,z^* \rangle \leq f(y),\,\,\forall y\in E\rbrace
\end{equation*}
for all $x\in E$. It is known that $\partial f$ is a maximal monotone operator by Rocfellar [21]. Let $C$ be a nonempty, closed and convex subset of $E$ and $i_C$ be the indicator function of $C$, i.e.,
\begin{equation*}
i_C(x)= \left\{
\begin{array}{r}
0 \quad x\in C,\\
\,\infty \quad x\notin C.
\end{array} \right.
\end{equation*}
 Then $i_C$ is a proper, lower semicontinuous and convex function on $E$ and hence, the subdifferential $ \partial i_C$ of $i_C$ is a maximal monotone operator. Therefore, the generalized resolvent $j_\lambda $ of $ \partial i_C$ for $\lambda >0$ is defined as follows:
 \begin{equation}
 J_\lambda x=(J+\lambda \partial i_C)^{-1}Jx, \quad \forall x\in E.
 \end{equation}
 For any $x\in E$ and $u\in C$, the following relations are hold:
 \begin{align*}
 u=&J_\lambda x \Longleftrightarrow Jx\in Ju+\lambda \partial i_Cu\\ &\Longleftrightarrow \frac{1}{\lambda}(Jx-Ju)\in  \partial i_Cu\\&\Longleftrightarrow i_Cy\geq \langle y-u,\frac{1}{\lambda}(Jx-Ju)\rangle +i_Cu,\,\,\,\forall y\in E\\
 &\Longleftrightarrow 0\geq \langle y-u,\frac{1}{\lambda}(Jx-Ju)\rangle ,\,\,\,\forall y\in C\\
 &\Longleftrightarrow \langle y-u,Jx-Ju\rangle \leq 0,\,\,\,\forall y\in C\\
 &\Longleftrightarrow u=\Pi_Cx.
 \end{align*}
 Next, using Theorem \eqref{aaa}, a strong convergence theorem for finding minimizers of convex functions in two Banach spaces is demonstrated.

\begin{theorem}
Let $E$ and $F$ be two $2$-uniformly convex and uniformly smooth real Banach spaces that satisfies the Opial condition and let $J_E$ and $J_F$ be the duality mappings on $E$ and $F$, respectively. Let $C$
and $Q$ be nonempty, closed and convex subsets of $E$ and $F$ respectively. Let $A:E\longrightarrow F$ be a bounded linear operator such that $A\neq0$ and with the adjoint operator $A^*$. Suppose that $S:C\longrightarrow E$ be a nonexpansive mapping and $\lbrace T_i \rbrace_{i=1}^{\infty} :C\longrightarrow C$ a family of nonexpansive mappings. For every $n\in \mathbb{N}$, let $W_n$ be a $W-mapping$ generated by Definition \ref{quu}. Let $x_1\in C$ and let $\lbrace x_{n}\rbrace$, $\lbrace u_{n}\rbrace$ and $\lbrace y_{n}\rbrace$ be the sequences generated by
  \begin{align*}
  \begin{cases}
  u_n=J_{E}^{-1}\Big((1-\alpha_n)J_{E}x_n+\alpha_n J_{E}\Pi_{C} J_{E}^{-1}(\sigma_n J_{E}W_n x_n+(1-\sigma_n )J_{E}Sx_n)\Big);\\
  z_n=\Pi_C(u_n+e_n);\quad \quad\omega_n=\Pi_Q(Az_n);\\
  y_n=\Pi_C J_{E}^{-1}\Big(J_{E}z_n-\gamma A^*J_{F}(Az_n-\omega_n)\Big);\\
  C_n=\lbrace z\in C,\big\langle \omega_n-Az,J_{F}(Az_n -\omega_n)\big\rangle\geq 0\rbrace;\\
  D_n=\lbrace z\in E,\phi_{E}(z,z_n)\leq \phi_{E}(z, u_n+e_n)\rbrace;\\
  Q_n=\lbrace z\in E,\langle x_n -z, J_{E}x_1-J_{E}x_n \rangle\geq 0 \rbrace;\\
  x_{n+1}=\Pi_{C_n\cap Q_n\cap D_n}x_1,\quad \forall n \in \mathbb{N}.\\
  \end{cases}
  \end{align*}
 where $0<\gamma<\frac{2}{c\parallel A\parallel^{2}}$ with $c>0$.
   Let $ \lbrace\alpha_n\rbrace, \lbrace\sigma_n\rbrace $ be real sequences in  $(0,1)$ satisfying the conditions
  \begin{itemize}
  \item [{\rm($i$)}] $\displaystyle\lim_{ n\rightarrow \infty}\alpha_n = 0$;
  \item [{\rm($ii$)}] $\displaystyle\lim_{n\rightarrow \infty}\dfrac{\Vert J_Ex_n-J_E u_n\Vert}{\alpha_n}=0$;
  \item [{\rm(iii)}] $\displaystyle\lim_{n\rightarrow \infty}\sigma_n =1$;
  \end{itemize}
   and the error sequence $\lbrace e_n\rbrace \subseteq E$ such that
  \begin{itemize}
    \item [{\rm($iv$)}] $\displaystyle\lim_{n\rightarrow \infty}\Vert e_n\Vert =0$.
  \end{itemize}
   Let $\Omega=C\cap A^{-1}Q\cap\big(\bigcap_{i=1}^{\infty}Fix(T_i)\big)$.
  Suppose that one of the following two conditions is hold:
  \begin{itemize}
    \item [(v)] the sequence $\{x_n\}$ is bounded,\\
    or
    \item [(vi)] $\Omega \neq \emptyset$.
  \end{itemize}
  Then
  \begin{enumerate}
    \item [(a)]$\Omega \neq \emptyset$ if and only if the sequence $\{x_n\}$ is bounded,
    \item [(b)]The sequence $\lbrace x_n\rbrace$ converges strongly to a point $\omega_0\in \Omega$ where $\omega_0 =\Pi_\Omega x_1$.
  \end{enumerate}
  \end{theorem}
   Next, Theorem \ref{aaa} will be illustrated by an example:\\
   \textbf{A numerical example} Let $E=F=\mathbb{R}$, the set of real numbers, with the inner product defined by $\langle x,y\rangle =xy, \forall x,y\in \mathbb{R}$, and usual norm $\vert\cdot\vert$. Suppose that $C=[0,1]$ and the mapping $A:\mathbb{R}\rightarrow \mathbb{R}$ is defined by $A(x)=-2x,\forall y\in \mathbb{R}$. Let $T_i:C\rightarrow C$ be the identity function for each $i\in \mathbb{N}$ and hence  the mapping $W_n: C\rightarrow C$ is the identity function for each $n\in \mathbb{N}$. Also suppose that $S:C\rightarrow \mathbb{R}$ is the identity function. Let $M_1, M_2:\mathbb{R}\rightarrow 2^{\mathbb{R}}$ be defined by $M_1(x)=\{2x\},\forall x\in \mathbb{R}$ and $M_2(y)=\{3y\}, \forall y\in \mathbb{R}$. Then $M_1^{-1}0\cap A^{-1}(M_2^{-1}0)\subseteq C$ and $\Omega \neq \emptyset$. Let $\lbrace\alpha _n\rbrace$ and $\lbrace\sigma _n\rbrace$ be arbitrary real sequences in $(0, 1)$ such that $\displaystyle\lim_{n\rightarrow \infty}\alpha_n=0$ and $\displaystyle\lim_{n\rightarrow \infty}\sigma_n =1$. Let $\lambda_n=\mu_n=0.25, \:e_n=n^{-1}$ for each $n\in \mathbb{N}$, and $\gamma=0.1$. Then the sequences $\lbrace x_n\rbrace, \lbrace u_n\rbrace, \lbrace z_n\rbrace, \lbrace \omega_n\rbrace$ and $\lbrace y_n\rbrace$ generated by \eqref{mnasd} az follows: given initial value $x_1\in C$
   \begin{align*}
  \begin{cases}
  u_n=x_n\\
  z_n=\dfrac{2}{3}(x_n+n^{-1});\quad\quad\omega_n=\dfrac{-16}{21}(x_n+n^{-1});\\
  y_n=\Pi_C\big(\dfrac{116}{210}(x_n+n^{-1})\big);\\
  C_n=\big\lbrace z\in C,z\leq \dfrac{16(x_n+n^{-1})}{42}\big\rbrace;\\
  D_n=\big\lbrace z\in E,z\leq \dfrac{5x_n+2n^{-1}}{6}\big\rbrace;\\
  Q_n=\lbrace z\in E,(x_n-z)(x_1-x_n)\geq 0\rbrace;\\
  x_{n+1}=\Pi_{C_n\cap Q_n\cap D_n}x_1,\quad\forall n\in \mathbb{N}.\\
  \end{cases}
  \end{align*}

\vspace{0.1in}
\hrule width \hsize \kern 1mm
\end{document}